\documentclass[12pt]{amsart}

\usepackage{amsfonts, amssymb, amscd}

\def\Bx             {{\rm Box}}
\def\Bcx           {{\rm Bx}}

\def\SSigma         {{\bf \Sigma}}

\def\ZZ                 {{\mathbb Z}} 
\def\PP                {{\mathbb P}} 
\def\RR                 {{\mathbb R}} 
\def\CC                 {{\mathbb C}} 
\def\QQ                 {{\mathbb Q}}

\def\Aa    {{\mathcal A}}

\newtheorem{lemma}{Lemma}[section] 
\newtheorem{theorem}[lemma]{Theorem} 
\newtheorem{corollary}[lemma]{Corollary} 
\newtheorem{proposition}[lemma]{Proposition} 
\theoremstyle{definition} 
\newtheorem{definition}[lemma]{Definition} 
\newtheorem{example}[lemma]{Example}

\newtheorem{remark}[lemma]{Remark} 
\theoremstyle{remark} 
\newtheorem*{proof*}{Proof} 

\title[Toric DM stacks and the bbGKZ system]{Toric Deligne-Mumford stacks and the 
better behaved version of the GKZ hypergeometric system}
\author[R. Paul Horja]{R. Paul Horja}

\address{Department of Mathematics\\ 
Oklahoma State University \\ 
Stillwater, OK 74078, USA\\
horja@math.okstate.edu}

\thanks{The author
was partially supported by the NSA grant MDA904-10-1-0190.}

\begin{document}

\begin{abstract}
We generalize the combinatorial description of the orbifold (Chen--Ruan)
cohomology and of the Grothendieck ring of a 
Deligne--Mumford toric stack and its associated stacky fan
in a lattice $N$ in the presence of a
deformation parameter $\beta \in N \otimes \CC.$
As an application, we construct a topological mirror symmetry
map that produces a complete system of $\Gamma$--series solutions
to the better behaved version of the GKZ hypergeometric system
for $\beta \in N \otimes \CC.$
\end{abstract}

\maketitle

\section{Introduction}

This note is a continuation of some recent joint
work of the author with Lev Borisov \cite{BH11} on the better behaved
version of the Gelfand--Kapranov--Zelevinsky hypergeometric system. 
The first two sections after the introduction, which may be of 
independent interest, contain 
a construction of certain modules over the Stanley--Reisner (SR) cohomology and
describe their Grothendieck ring interpretation in the context of 
toric Deligne-Mumford stacks. The notion of SR--cohomology 
and the computation of the Grothendieck ring 
of a smooth Deligne--Mumford stack have been 
discussed in \cite{BH1}; when the stack is projective,
the SR--cohomology is the same as the orbifold (Chen--Ruan)
cohomology. For a given stacky fan $\SSigma$ in a lattice $N$ and a
parameter $\beta \in N \otimes \CC,$ we give a combinatorial
description (cf. Definition \ref{def:coh})
of the deformed SR--cohomology module $H\ZZ[\SSigma;\beta]$, and of the 
deformed Grothendieck ring $K\CC[\SSigma;\beta]$ (cf. Definition \ref{def:kth1}). 
Not surprisingly, the two deformed notions 
are related by a Chern character--type map (Theorem \ref{thm:iso1}) and, 
in the case of an 
integral parameter $\beta \in N,$  we recover the geometric calculations
of SR--cohomology and of the Grothendieck ring 
of a toric Deligne--Mumford stack, respectively.
Moreover, Proposition
\ref{prop:rel} shows that it is possible to reduce the discussion for 
a complex parameter $\beta \in N \otimes \CC$ to the more geometric 
case of the real parameter $\Re \beta + \delta \Im \beta \in N \otimes \RR$
for all sufficiently small real numbers $\delta > 0.$
It is tempting to conjecture that the 
constructions presented here and the presence of certain ``walls" in 
the space of parameters $\beta \in N \otimes \RR$ 
(cf. Example \ref{ex:ex})
have an equivariant algebra--geometric
interpretation, but we postpone the discussion of such
speculations for future work.

As an application, we present a 
construction of a complete 
set of $\Gamma$ series solutions to the better behaved GKZ system 
in the case $\beta \in N \otimes \CC.$ For the classical GKZ 
system, such solutions have been obtained in various degrees of generality 
in the works of \cite{HLY}, \cite{stienstra}, \cite{BH2}, as well
as in the book \cite{SST}.

In a celebrated paper, Gelfand, Kapranov and Zelevinsky \cite{GKZ}
considered a system of linear partial differential equations 
associated to the set $\Aa$ and any parameter 
$\beta \in N \otimes \CC.$ Besides being an interesting 
object to study in itself, the GKZ system has become a 
fundamental tool in mirror symmetry after the discovery
by \cite{Batyrev} that the Picard-Fuchs equations 
associated to the periods of Calabi--Yau hypersurfaces 
in toric varieties are in fact of GKZ type. 
If the toric ideal associated to $\Aa$ is not Cohen-Macaulay, 
there are non-generic values of $\beta$ for which the rank 
of the solution set jumps. A thorough analysis of this 
property is contained in the work \cite{MMW}. However, this phenomenon 
brings with it a degree of non-functoriality in
some of the applications of the GKZ system to 
the more recent categorical questions in 
algebraic geometry and homological mirror symmetry.

This discrepancy constituted the starting point of 
our recent joint work \cite{BH11}, where we proposed 
a {\it better
behaved} version of the GKZ system whose
space of solutions has the expected
number of solutions in all cases. We framed the definition 
in a context where the lattice is replaced
by a finitely generated abelian group $N.$
and the set $\Aa$ determining the PDE system
is replaced by an 
$n$-tuple of elements of $N,$ 
with possible repetitions. 

In section \ref{sec:bbgkz} of this paper, we only consider the case
when $N$ is a lattice of rank $d,$
and 
$\Aa=\{ v_1, \ldots, v_k\}$ is a set. We assume that the 
elements of $\Aa$ generate the lattice as an abelian group 
and that there exists a group homomorphism ${\rm deg} : N \to \ZZ$
such that ${\rm deg} (v_i)=1$ for any element $v_i \in \Aa.$
Let $K \subset N \otimes \RR$ be the cone over the 
polytope $\Delta={\rm conv}(\Aa),$ the convex hull of the elements
of $\Aa.$ 

For any parameter $\beta$ in $N \otimes \CC,$ the {\it better 
behaved Gelfand--Kapranov--Zelevinsky system} consists of the system of
partial differential equations
\begin{equation}\label{eq:gkz1}
\partial_i\Phi_v=\Phi_{v+v_i},~{\rm~for~all~}v\in K \cap N,~i\in\{1,\ldots,k\},
\end{equation}
and the linear equations

\begin{equation}\label{eq:linear}
\sum_{i=1}^k g(v_i) x_i\partial_i \Phi_v=g(\beta-v) \Phi_v, 
{\rm~~for~all~} g \in M={\rm Hom}(N, \ZZ),~v\in K \cap N. 
\end{equation}
A solution to the better behaved GKZ system 
is then a
sequence of functions of $k$ variables 
$\big(\Phi_v(x_1, \ldots, x_k)\big)_{v \in K \cap N}.$
Alternatively, it can 
be viewed as a function in $k$ 
variables  
$$\Phi_K(x_1, \ldots, x_k):=\sum_{v \in K \cap N} \Phi_v (x_1, \ldots, x_k) [v]$$ 
with values in the completion $\CC[K \cap N]^\wedge$
of the ring  
$\CC[K \cap N].$ In section \ref{sec:bbgkz} of this note, we use the 
constructions of the first section to obtain 
${\rm vol} (\Delta)$ linearly independent $\Gamma$--series
solutions for the better behaved GKZ system which are
analytic in an non--empty open set $U_\SSigma$ corresponding
to any regular simplicial fan structure supported on $K$
whose one dimensional cones are generated by elements of $\Aa.$
The solutions are obtained with 
the help of a ``topological mirror symmetry map"
(cf. Theorem \ref{thm:gammanontor}). 
In conjunction
with the  results of \cite{BH11}, this construction
gives an explicit complete set of solutions in the general case
of a finitely generated abelian group $N,$ a collection 
$\Aa$ with possible repetitions and any parameter $\beta
\in N \otimes \CC.$

\section{Deformed modules over the SR--cohomology of a reduced toric DM stack}
\label{sec1}

In this section, 
$N \simeq \ZZ^d$ is a lattice of rank $d,$
$\SSigma=(\Sigma, \{v_i\}_{1 \leq i \leq k})$ is
a simplicial stacky fan; we assume that
any of the elements $v_i \in N$ generates
a ray (one dimensional cone) of the rational
simplicial fan $\Sigma.$

\

{\it Notation.}
As it is customary,  we denote by
$\{a\}$, $0 \leq \{a\} <1,$ the 
fractionary part of the real number $a.$ 
For any cone 
$\sigma$ of the fan $\Sigma,$ we will denote by 
$I(\sigma) \subset \{1, 2, \ldots, k\}$ 
the subset consisting of those indices $i$ such that 
$v_i$ generates a ray of the cone $\sigma.$ 
To any subset $I \subset \{1, 2, \ldots, k\},$ 
we associate the cone $\sum_{i \in I} \RR_{\geq 0} v_i$
and we denote it by $\sigma(I).$ We denote by $\sigma$
the cone as well as its support in $N \otimes \RR,$ and 
by $\Sigma$ the union of the supports of all the cones
of the fan.


\

For pedagogical reasons, we begin by assuming  that  
$\beta \in N \otimes \RR$ is the fixed parameter.  
For a maximal cone $\sigma$ of $\Sigma,$ 
let
$\Bx (\sigma; \beta)$ denote the finite subset of 
elements  
$c \in \sigma \subset N \otimes \RR$ 
of the form $n + \beta,$
for some $n \in N,$ 
such that
$$c=n+\beta= \sum_{i=1}^k \alpha_i v_i,$$
with $0 \leq \alpha_i <1,$ for all $i, 
1 \leq i \leq k,$ and 
$\alpha_i=0,$ if $i \notin I(\sigma).$

Let $\Bx (\SSigma; \beta)$ be the union of the sets
$\Bx (\sigma; \beta)$ for all the maximal cones
$\sigma$ of
$\Sigma.$ For any element 
$c$ in $\Bx (\SSigma,\beta),$ the 
minimal cone of $\Sigma$ containing 
$c$ is denoted by $\sigma(c)$ and it 
is called the {\it support cone of $\alpha.$} 
When $\beta \in N,$ the set 
$\Bx (\SSigma; \beta)$ coincides with the
usual set 
$\Bx (\SSigma) \subset N$ from the
theory of toric DM stacks.

Let $\ZZ[\SSigma]$ be the deformed semigroup 
ring defined as the free abelian group with the basis elements 
$[n]$ for all $n \in N \cap \Sigma,$ and the product 
$[m] \cdot [n]= [m+n],$ when $m$ and $n$ belong
to some cone of the fan, and $0$ otherwise. 
This will be the ground ring in our discussion, so we 
call it $R.$ 

\begin{definition}\label{def:1}
The $R$--module $\ZZ[\SSigma; \beta]$ 
is generated as a
free abelian group by the basis elements
$[n+\beta]$ for all $n \in N$ such that
$n+\beta \in \Sigma.$
The module structure is given by 
the product $[n^\prime] \cdot [n+\beta ]:=[n^\prime+n + \beta],$ 
if there exists a cone of the fan containing $n^\prime$ and $n+\beta,$
and $0$ otherwise. 
\end{definition}

\begin{remark}\label{rem:boxlimit}
The underlying theme of this section has to do with
the change in the structure of the modules 
$\ZZ[\SSigma; \beta]$ when the parameter $\beta$ is allowed to vary
in $N \otimes \RR.$ The module structure is related
to the combinatorics of the set $\Bx (\SSigma; \beta),$ and the following example features 
one of the subtleties that has to be considered.

In the lattice $N=\ZZ^2$ consider the fan 
$\Sigma$ whose two maximal cones are generated by 
the pairs of vectors $\{v_1,v_2\}$ and $\{v_2,v_3\},$ where 
$v_1=(1,0), v_1=(1,1), v_2=(1,2).$ Let $\beta= (\delta, 0) \in N \otimes \RR,$
with $0 < \delta < 1.$  We see that $\Bx (\SSigma;\beta)=
\{ (\delta,0), (1+ \delta,2) \},$ and that the pointwise
limit in $N \otimes \RR \simeq \RR^2$ of 
$\Bx (\SSigma;\beta)=
\{ (\delta,0), (1+ \delta,2) \}$ is the set $\{ (0,0), (1,2) \}.$ 
On the other hand, it is certainly the case that 
$(\delta,0) \to (0,0)$ when $\delta \to 0,$ and 
$\Bx (\SSigma;(0,0))=\{ (0,0) \}.$ It is possible
to gain a general understanding of such a limiting 
process by defining a $\gamma$--deformed version of the
set $\Bx (\SSigma; \beta)$ for any $\gamma \in 
N \otimes \RR,$ but we leave the details of this
construction as an exercise. Moreover, 
a related  issue is the subject of Corollary \ref{cor:rel} which 
is proved below. 
\end{remark}

We now move to the more general
case $\beta \in N \otimes \CC.$
We  consider a version of 
the set $\Bx (\SSigma;\beta) \subset N \otimes \RR \simeq \RR^d$
adapted to this situation.
Given a maximal cone
$\sigma$ of the fan
$\Sigma,$ let
${\rm Bx} (\sigma; \beta) \subset \CC^k$ be defined as the finite set
consisting of elements $\alpha=(\alpha_i)_{1 \leq i \leq k}$ 
such that
$$
n + \beta = \sum_{i=1}^k \alpha_i v_i,$$
for some $n \in N,$ 
with $0 \leq \Re \alpha_i <1,$ for all $i, 
1 \leq i \leq k,$ and 
$\alpha_i=0,$ if $i \notin I(\sigma).$
It is clear that, in such a situation, we have that
$n + \Re \beta \in \sigma.$ The set
$\Bcx (\SSigma; \beta) \subset \CC^k$ is then defined as the 
union of the sets
$\Bcx (\sigma; \beta)$ for all the maximal cones
$\sigma$ of $\Sigma.$ 
For any element 
$\alpha$ in $\Bcx (\SSigma,\beta),$ the 
support cone $\sigma(\alpha)$ of 
$\alpha$ is now defined as the
cone of $\Sigma$ generated by all the $v_i$ 
such that $\alpha_i \not=0.$ 

Note that the cardinality of the finite set 
$\Bcx(\sigma; \beta)$ is given by
the index of the sublattice generated
by the generators of the cone $\sigma$ in $N.$ 
If $\Im \beta$ is generic, the
sets $\Bcx(\sigma; \beta)$ are disjoint, otherwise
$\beta$ determines their intersections. 
When $\beta \in N \otimes \RR,$
the definition of
$\Bcx (\SSigma; \beta)$ is equivalent to the
definition of $\Bx (\SSigma;\beta)$
given earlier in this section, 
since in that case there is a one-to-one 
correspondence between the element
$\alpha=(\alpha_i) \in \Bcx (\sigma; \beta)$
and the corresponding element
$c=\sum_{i=1}^k \alpha_i v_i \in \sigma \cap (N + \beta).$
For a general $\beta \in N \otimes \CC,$
it may happen that for distinct elements 
$\alpha, \alpha^\prime \in \Bcx(\SSigma; \beta),$
we have that 
$\sum_{i=1}^k \alpha_i v_i =\sum_{i=1}^k \alpha^\prime_i v_i.
$ In order to account for such occurences, we have to adjust 
Definition \ref{def:1} accordingly.

\begin{definition}\label{def:2}
Given $\beta \in N \otimes \CC,$
the $R$--module $\ZZ [\SSigma; \beta]$ is 
defined as the free abelian group with the basis elements
$[n+\beta, \alpha]$ for all $n \in N$
and $\alpha=(\alpha_i) \in \Bcx (\sigma; \beta)$ 
for some maximal cone $\sigma$ of $\SSigma$
such that 
$$n + \beta \in
\sum_{i=1}^k \alpha_i v_i + \sum_{i \in I(\sigma)} \ZZ_{\geq 0} v_i.$$
The module structure is then defined by 
the product $[n^\prime] \cdot [n + \beta, \alpha]:=[n^\prime+n+\beta,\alpha^\prime],$ 
if there exists a maximal cone $\sigma$ of $\SSigma$ such that
$n^\prime \in \sigma \cap N,$ 
$\alpha \in \Bcx(\sigma; \beta),$ and 
$$n + \beta  \in
\sum_{i=1}^k \alpha_i v_i + \sum_{i \in I(\sigma)} \ZZ_{\geq 0} v_i, \
n^\prime +n +\beta 
\in 
\sum_{i=1}^k \alpha_i^\prime v_i + 
\sum_{i \in I(\sigma)} \ZZ_{\geq 0} v_i.$$
If no such maximal cone $\sigma$ exists, then the 
product is $0$ by definition. 
\end{definition}

If the product is non-zero, for any maximal cone
$\sigma$ as above we have that 
$\alpha^\prime \in \Bcx(\sigma;\beta).$
 It is easy to see that
the module structure is well defined, i.e. 
the product structure does not depend on the choice of a maximal cone 
$\sigma$ in the above definition.

Let $S_\SSigma$ be the subring of $R$ generated  
by the elements $[v_i].$ We then have the 
the following direct sum
decomposition of the $R$--module
$\ZZ [\SSigma;\beta]$ 
into $S_\SSigma$--submodules: 
\begin{equation}\label{decomp:ds}
\ZZ [\SSigma;\beta] = \bigoplus_{\alpha \in \Bcx (\SSigma;\beta)}
S_{\SSigma} \cdot [\sum_ {i=1}^k \alpha_i v_i, \alpha].
\end{equation}

\begin{remark}\label{rem:11}
When $\beta \in N \otimes \RR,$
the $R$--module structures introduced in definitions \ref{def:1} and
\ref{def:2} are identified under the map $[n+\beta] \mapsto
[n + \beta, \alpha].$ Indeed, given $n + \beta \in \Sigma,$
we consider the minimal cone 
$\sigma(n+\beta)$ of $\SSigma$ containing 
$n + \beta,$ and we write that
$n + \beta= \sum_{i=1}^k q_i v_i$ with 
$q_i \in \RR_{\geq 0},$ and 
$q_i = 0,$ if $i \notin I(\sigma(n+\beta)).$
The element $\alpha=(\alpha_i) \in \Bcx (\SSigma;\beta) 
\subset \RR^k,$ is then uniquely defined by setting
$\alpha_i: = \{ q_i \},$ for all $i.$ 
\end{remark}

As the following statement shows, the real and complex 
cases of the parameter $\beta$ are closer than one might have 
suspected at first glance.  

\begin{proposition}\label{prop:rel}
For any $\beta \in N \otimes \CC.$
there exists an
isomorphism of $R$-modules
$
\ZZ[\SSigma; \beta] \simeq
\ZZ [\SSigma; \Re \beta + \delta \Im \beta],
$
for all
sufficiently small $\delta >0.$
\end{proposition}

\begin{proof} 
The combinatorics that
determines the structure 
of 
$\ZZ[\SSigma; \Re \beta + \delta \Im \beta]$
as an $R$--module
``stabilizes" when $\delta >0$ is sufficiently small.
We must show that this stabilized $R$--module 
structure is isomorphic to the $R$--module 
$\ZZ [\SSigma; \beta].$ 

We define the map $\varphi_\delta : \ZZ [\SSigma; \beta] \to \ZZ[\SSigma; \Re \beta + \delta \Im \beta]$
by setting $\varphi_\delta ([n+\beta, \alpha]):= [n+\Re \beta + \delta \Im \beta, \alpha_\delta],$
where 
$(\alpha_\delta)_i : = \{ \Re \alpha_i + \delta \Im \alpha_i\}.$ Note 
that the map is well defined, since the support cone $\sigma(\alpha_\delta)$
is a subcone of $\sigma(\alpha),$ so any maximal cone which contains
$\sigma(\alpha)$ will also contain $\sigma(\alpha_\delta).$ 
In fact, if $\delta >0$ is small enough so that
$\Re \alpha_i + \delta \Im \alpha_i <1$ for all $i,$  
the support cones
$\sigma(\alpha)$ and $\sigma(\alpha_\delta)$ coincide. 

Let $n, n^\prime \in N$ and $\sigma$ a maximal 
cone such that
$$
n^\prime \in \sum_{i=1}^k  \gamma_i v_i  + \sum_{i \in I(\sigma)} \ZZ_{\geq 0} v_i, 
n+\beta \in \sum_{i=1}^k \alpha_i v_i  + \sum_{i \in I(\sigma)} \ZZ_{\geq 0} v_i, 
$$
where $\gamma=(\gamma_i) \in \Bcx (\SSigma; 0) \subset \QQ^k,$ 
$\alpha=(\alpha_i) \in \Bcx (\SSigma; \beta) \subset \CC^k.$
In the $R$--module $\ZZ[\SSigma;\beta],$ we then have that 
$[n^\prime] \cdot [n+\beta, \alpha] = [n^\prime+n +\beta, \alpha^\prime],$
where
$\alpha^\prime_i= \{ \gamma_i + \Re \alpha_i\}+  \sqrt{-1} \Im \alpha_i,$
for all $i.$ We see that
$$
\varphi_\delta ([n^\prime] \cdot [n+\beta, \alpha])=
[n^\prime +n + \Re \beta + \delta \Im \beta, \alpha^\prime_\delta]=
[n^\prime] \cdot \varphi_\delta ([n+\beta,\alpha]),$$
where 
$(\alpha^\prime_\delta)_i=\{ \gamma_i + \Re \alpha_i + \delta \Im \alpha_i \}.$
This shows that $\varphi$ is an $R$--module homomorphism. 

It is clear that 
$\varphi_\delta$ is a monomorphism. 
We now show that $\varphi_\delta$ is an epimorphism provided that $\delta >0$ is sufficiently small.
Let
$[n+ \Re \beta + \delta \Im \beta, \alpha_\delta]$ be an element
in the ``stabilized" module $\ZZ[\SSigma; \Re \beta + \delta \Im \beta].$ 
We assume that 
$\delta>0 $ is small enough so that 
the support cone $\sigma(\alpha_\delta)$ is constant
and contained in a maximal cone $\sigma$ such that
$$
n+ \Re \beta + \delta \Im \beta=
\sum_{i \in I(\sigma)} 
((\alpha_\delta)_i +p_i) v_i,
$$
with $p_i \in \ZZ_{\geq 0}$ independent of $\delta.$
Since $n + \Re \beta \in \sigma,$ there exists
$\alpha \in \Bcx(\sigma;\beta)$ such that
$$
n + \beta= \sum_{i \in I(\sigma)} 
(\alpha_i +p^*_i) v_i,
$$
with
$p_i^* \in \ZZ_{\geq 0}.$ We are dealing with
linear combinations over the generators of
a maximal cone, so it must be the case that,
for all $i,$ 
$(\alpha_\delta)_i : = \{ \Re \alpha_i + \delta \Im \alpha_i\}.$
Hence, $\varphi_\delta ([n+ \beta, \alpha])=
[n+ \Re \beta + \delta \Im \beta, \alpha_\delta],$
and this ends the proof of the proposition.
\end{proof}

\begin{corollary}\label{cor:rel}
For any $\beta \in N \otimes \CC$
and any sufficiently small $\delta >0,$
there exists a triple one-to-one correspondence 
$\alpha \leftrightarrow \alpha_\delta \leftrightarrow
c(\alpha_\delta)$ 
among the finite sets 
$\Bcx (\SSigma;\beta) \subset \CC^k,$ 
$\Bcx (\SSigma;\Re \beta+ \delta \Im \beta) \subset \RR^k$
and $\Bx (\SSigma;\Re \beta+ \delta \Im \beta) 
\subset N \otimes \RR,$ such that the support
cones $\sigma(\alpha),$
$\sigma(\alpha_\delta)$ and $\sigma(c(\alpha_\delta))$ coincide.
\end{corollary}
The ``pointwise"
limit (in $\RR^k$) of the set $\Bcx (\SSigma;\Re \beta+ \delta \Im \beta)$
is closely related, but in general it is 
not {\it equal} to the set $\Bcx (\SSigma; \Re \beta).$
This observation is related to the issue discussed in Remark
\ref{rem:boxlimit}.

\

We choose a basis $\{ g_1, \ldots, g_d \}$ of the
dual lattice $M= {\rm Hom}(N,\ZZ),$ $d={\rm rk} N.$
Let $Z_j \in S_{\SSigma} \subset \ZZ[\SSigma]$ be
given by
$$Z_j: = \sum_{i=1}^k g_j(v_i) [v_i], \ 1 \leq j \leq d,$$
and we denote by $Z:=(Z_1, \ldots, Z_d) \subset S_\SSigma$ the ideal
generated by $Z_1, \ldots, Z_d.$

\begin{definition}\label{def:coh}
We call 
the quotient $R$--module
$$
H\ZZ (\SSigma;\beta):=
\ZZ [\SSigma;\beta]/Z \cdot \ZZ[\SSigma;\beta],$$
the {\it deformed SR--cohomology} modules associated to the stacky fan 
$\SSigma$ and the element
$\beta \in N \otimes \CC.$
\end{definition}
Of course, the constructions presented so far can be performed
with the ring $\ZZ$ replaced by $\QQ$ or $\CC.$ When 
$\beta \in N,$ the deformed SR--cohomology module coincides
with the SR--cohomology ring $R/Z \cdot R$ which is the
SR--cohomology ring associated to the toric DM--stack
$\PP_\SSigma$ as described by
\cite[Definition 3.1]{BH1}. If in addition the stack is 
projective, then the SR--cohomology ring with $\QQ$ coefficients
is isomorphic to
the orbifold cohomology ring of $\PP_\SSigma$
(cf. \cite[Theorem 1.1]{BCS}).

Since $Z_j \in S_\SSigma,$ the direct sum decomposition
(\ref{decomp:ds})
descends to the following decomposition into a sum of $S_\SSigma/Z$--submodules:
\begin{equation}\label{eq:dec1}
H\ZZ (\SSigma;\beta)
\simeq \bigoplus_{\alpha \in \Bcx (\SSigma;\beta)}
(S_{\SSigma}/ Z)  \cdot [\sum_{i=1}^k \alpha_i v_i, \alpha].
\end{equation}
We should also note
that the ring $S_{\SSigma}/ Z$ is 
isomorphic to the Chow
ring of the toric Deligne--Mumford
stack $\PP_\SSigma$ which is the
same as the Chow ring of the
toric variety $\PP_\Sigma,$ i.e.
$A^\star(\PP_\Sigma)
\simeq A^\star(\PP_\SSigma) \simeq S_{\SSigma}/ Z$
(cf. \cite[Lemma 5.1]{BCS}).

\section{Deformations of the Grothendieck ring of
a reduced toric DM stack}\label{sec:kth}

\begin{definition}\label{def:kth1}
The {\it deformed Grothendieck ring}
$K\CC(\SSigma;\beta)$
associated to the stacky fan 
$\SSigma=(\Sigma, \{v_i\}_{1 \leq i \leq k})$ and the
parameter $\beta \in N \otimes \CC$ is defined
to be the quotient of the Laurent polynomial ring
$\CC[R_1^{\pm 1},\ldots,R_k^{\pm 1}]$
by the ideal $U+V,$ where $U$ is 
the ideal generated by the relations
$$\displaystyle \prod_{i \in I} (1 -R_i)= 0,$$ 
for any set 
$I \subset \{1, \ldots, k\}$ such that $v_i, i \in I,$ 
do not generate a cone of $\Sigma,$ and 
$V$ is the ideal generated by the relations
$$\prod_{i=1}^k R_i^{f(v_i)}= e^{2 \pi \sqrt{-1} f(\beta)},$$ 
for any linear function $f: N \to \ZZ.$ 
\end{definition} 

When $\beta \in N,$ the definition can be given over $\ZZ$
and in that case, the ring is isomorphic to the Grothendieck
ring $K_0 (\PP_\SSigma)$ of the associated reduced
toric Deligne--Mumford stack $\PP_\SSigma$
(cf. \cite[Theorem 4.10]{BH1}).

\begin{proposition} \label{prop:kth1}
The ring 
$K\CC(\Sigma;\beta)$ is Artinian. 
Its maximum ideals are in one-to-one correspondence with 
the elements of $\Bcx ({\bf\Sigma};\beta) \subset \CC^k$ 
as follows: any element $\alpha=(\alpha_i)$ 
in $\Bcx ({\bf\Sigma};\beta)$ corresponds to the
maximal ideal determined by the $n$-tuple 
$y=(y_i) \in \CC^k$ with
$y_i= e^{2 \pi \sqrt{-1} \alpha_i}.$ 
\end{proposition}
\begin{proof} We need to solve for $y=(y_i) \in \CC^k$
such that 
$\prod_{i \in I} (1 -y_i)= 0,$ for any set 
$I \subset \{1, \ldots, k\}$ such that   
$\sigma(I)$ is not a cone in $\Sigma,$
and 
$
\prod_{i=1}^k y_i^{f(v_i)}= e^{2 \pi \sqrt{-1} f(\beta)}
$
for all linear $f : N \to \ZZ.$

Assume that $y=(y_i)$ is a solution.
The first set of equations 
implies the existence of a maximal cone
$\sigma$ such that $y_1=1$ for all
$i \notin I(\sigma).$ 
The second set of equations shows that all $y_i$
are nonzero. There exists a unique choice 
$(\alpha_i) \in \CC^k$
such that
$y_i= e^{2 \pi \sqrt{-1} \alpha_i}$ 
and $0 \leq \Re \alpha_i <1$ if $i \in I(\sigma),$
and $\alpha_i =0$ if  $i \notin I(\sigma).$
The second set of equations implies that 
$\sum_{i=1}^k \alpha_i v_i= n + \beta,$ for some $n \in N,$
so $(\alpha_i) \in \Bcx(\SSigma;\beta).$
Conversely, the definition of $\Bx(\SSigma; \beta)$
shows that, for any $(\alpha_i)\in \Bcx(\SSigma; \beta)$
the induced $y_i= e^{2 \pi \sqrt{-1} \alpha_i}$ 
provide a solution of the required system. It is 
easy to see that this is a one-to-one correspondence. 
It follows that
the ring is indeed Artinian, since it is
Noetherian of Krull dimension zero. 
\end{proof}

We denote by $(K\CC(\Sigma;\beta))_\alpha$ 
the localization of the Artinian ring 
$K\CC(\Sigma;\beta)$
at the maximal ideal corresponding to
some $\alpha \in \Bcx (\SSigma;\beta).$
The Artinian ring
$K\CC(\Sigma;\beta)$
is a direct sum of the Artinian local
rings 
$(K\CC(\Sigma;\beta))_\alpha.$
As the following result shows, 
this direct sum decomposition 
is the counterpart of the 
direct sum decomposition of the
deformed cohomology 
$H\ZZ(\SSigma;\beta)$
into 
$S_\SSigma/Z$--submodules as
described by formula (\ref{eq:dec1}).

Let 
$S_{\SSigma}(\alpha)$ denote the kernel 
of the $S_{\SSigma}$--module homomorphism
$S_\SSigma \to \ZZ[\SSigma;\beta]$
given by $[v] \mapsto
[v] \cdot [\sum_{i=1}^k \alpha_i v_i, \alpha].$
We have that
$$
S_{\SSigma}/S_{\SSigma}(\alpha) \simeq 
S_{\SSigma} \cdot [\sum_{i=1}^k \alpha_i v_i, \alpha],$$
as $S_{\SSigma}$--modules. In particular, we can write  
the following direct sum decomposition of deformed
cohomology into
$S_\SSigma/Z$--submodules:
$$
H\ZZ(\SSigma;\beta)
\simeq
\bigoplus_{\alpha \in \Bx (\SSigma;\beta)}
S_{\SSigma}/ (Z+S_\SSigma (\alpha)).
$$
As we noted before, the ring $S_{\SSigma}/ Z$ is isomorphic 
to the Chow ring of the 
toric Deligne-Mumford stack $A^\star(\PP_\SSigma).$

\begin{theorem}\label{thm:iso1} 
There exists a $\CC$--vector space isomorphism
$$
\phi:
K\CC(\SSigma;\beta)
\simeq
H\CC(\SSigma;\beta).$$
It is 
induced by $\CC$--algebra isomorphisms
$$\phi_\alpha
(K\CC(\SSigma;\beta))_\alpha
\simeq 
(S_{\SSigma} \otimes_\ZZ \CC)/ (Z + S_{\SSigma}(\alpha)),
$$
for each element
$\alpha=(\alpha_i) \in \Bcx (\SSigma;\beta),$ induced by the identifications
$
R_i= e^{2 \pi \sqrt{-1} \,\alpha_i + [v_i]},$ for all $i,$
$1 \leq i \leq k.$
\end{theorem}

\begin{proof} Let us define the ring homomorphism 
$\phi_\alpha : \CC[R_1^{\pm 1}, \ldots, R_k^{\pm 1}] \to
(S_{\SSigma} \otimes_\ZZ \CC)/ (Z + S_{\SSigma}(\alpha))$
by setting
$$
\phi_\alpha (1):=1, \ \phi_\alpha (R_i^{\pm 1}):= e^{\pm (2 \pi \sqrt{-1} \,\alpha_i + [v_i])}.$$

Firstly, 
we check that $\phi_\alpha$ actually descends to a ring homomorphism 
$K\CC(\SSigma;\beta)
\to 
(S_{\SSigma} \otimes_\ZZ \CC)/ (Z + S_{\SSigma}(\alpha)).$
For, if $f : N \to \ZZ$ is linear, then 
$$
\phi_\alpha (\prod_{i=1}^k R_i^{f(v_i)})=
e^{\sum_{i=1}^k (2 \pi \sqrt{-1} \, \alpha_i  + [v_i])f(v_i)}.
$$
But $\sum_{i=1}^k f(v_i) [v_i]$ is an element of the ideal $Z,$
and $\sum_{i=1}^k \alpha_i f(v_i)=f(n +\beta),$ for some $n \in N,$
hence 
$\phi_\alpha (\prod_{i=1}^k R_i^{f(v_i)})=
e^{2 \pi \sqrt{-1} f(\beta)},$ as required. Let
$I$ be a subset of
$\{1, \ldots, k\}$
such that the cone $\sigma(I)$ 
is not in $\Sigma,$
and let $J$ denote the nonempty set $I \setminus I(\sigma(\alpha)).$
We see that $\alpha_i=0$ for $i \in J,$ 
so
$\phi_\alpha (\prod_{i \in J} (1 -R_i))=
\prod_{i \in J} (1-e^{[v_i]}).$
But the rays generated by the vectors $v_i,$
$i \in J,$ and the rays of the cone $\sigma(\alpha)$ 
do not generate a cone in $\Sigma,$
so $\prod_{i \in J} [v_i] \cdot [\alpha]=0$
in $\ZZ[\SSigma;\beta].$ Hence,  
$\phi_\alpha (\prod_{i \in I} (1 -R_i))=
\phi_\alpha (\prod_{i \in J} (1 -R_i))=0$ in 
$(S_{\SSigma} \otimes_\ZZ \CC)/ (Z + S_{\SSigma}(\alpha)).$

Secondly, we must check that 
$\phi_\alpha$ descends further to a ring homomorphism 
$(K\CC(\SSigma;\beta) \otimes_{\ZZ} \CC)_\alpha 
\to (S_{\SSigma} \otimes_\ZZ \CC)/ (Z + S_{\SSigma}(\alpha)).$
We check this by showing that, for any Laurent polynomial 
$F(x_1^{\pm 1}, \ldots, x_k^{\pm 1})$ with complex coefficients 
which is not in the maximal ideal of the localization, i.e.
such that $F(e^{\pm 2 \pi \sqrt{-1} \, \alpha_1}, \ldots, e^{\pm 2 \pi \sqrt{-1}\, \alpha_k})
\not= 0,$ it is true that $\phi_\alpha(F(R_1^{\pm 1}, \ldots, R_k^{\pm 1}))$ is a 
unit in $(S_{\SSigma} \otimes_\ZZ \CC)/ (Z + S_{\SSigma}(\alpha)).$ But
$$
\phi_\alpha(F(R_1^{\pm 1}, \ldots, R_k^{\pm 1})=
F(e^{\pm (2 \pi \sqrt{-1} \,\alpha_1 + [v_1])}, \ldots, 
e^{\pm (2 \pi \sqrt{-1} \,\alpha_k + [v_k])}).
$$
Note that
the elements $[v_i]$ are nilpotent 
in the ring $S_{\SSigma}/ Z$
since they represent classes in the Chow ring 
$A^\star(\PP_\SSigma) \simeq S_{\SSigma}/ Z.$ 
Hence, the right hand side of the above equation is a polynomial 
with complex coefficients in the variables $[v_i]$
whose evaluation at $(0,\ldots,0)$ is non-zero, so 
its inverse 
$1/F(e^{\pm (2 \pi \sqrt{-1} \,\alpha_1 + [v_1])}, \ldots, 
e^{\pm (2 \pi \sqrt{-1} \,\alpha_k + [v_k])})$ is
also a polynomial in the variables $[v_i],$ i.e.
an element in $(S_{\SSigma} \otimes_\ZZ \CC)/ (Z + S_{\SSigma}(\alpha)).$
Therefore
$\phi_\alpha : 
(K\CC(\SSigma;\beta))_\alpha 
\to S_{\SSigma}/ (Z+S_{\SSigma}(\alpha)) \otimes_{\ZZ} \CC$
is indeed a well defined ring homomorphism.

We now define a ring homomorphism
$\psi_\alpha : S_{\SSigma} \otimes_\ZZ \CC \to
(K\CC(\SSigma;\beta))_\alpha.$
Any element in the maximal ideal of an Artinian local ring
is nilpotent. In particular, $e^{-2 \pi \sqrt{-1} \, \alpha_i} R_i$
is unipotent in 
$(K\CC(\SSigma;\beta))_\alpha.$ We set
$$
\psi_\alpha (1):=1, \ \psi_\alpha ([v_i]):=\log 
(e^{-2 \pi \sqrt{-1} \alpha_i} R_i).$$
The definition uses the usual expansion of the 
log function given by  $\log x= \sum_{j \geq 1} (-1)^{j+1}(x-1)^j/j.$

For any linear function $f : N \to \ZZ,$ we can see that
$$
\psi_\alpha (\sum_{i=1}^k f(v_i) [v_i])=
\log (e^{-2 \pi \sqrt{-1} \, \sum_{i=1}^k \alpha_i f(v_i)} 
\prod_{i=1}^k  R_i^{f(v_i)})=0,$$
since $\sum_{i=1}^k \alpha_i f(v_i)=f(n +\beta),$ for some $n \in N,$
and $\prod_{i=1}^k  R_i^{f(v_i)}=e^{2 \pi \sqrt{-1} \, f(\beta)}$
in $(K\CC(\SSigma;\beta))_\alpha.$
Let $v= \sum_{i\in I} p_i v_i,$ $p_i \in \ZZ_{> 0},$ be
an element in $N$ such that 
$[v] \cdot [\sum_{i=1}^k \alpha_i v_i,\alpha]=0$ in $\ZZ[\SSigma;\beta].$
From the definition of the product, we see that the rays of the cones
$\sigma(I)$ and $\sigma(\alpha)$ do not
generate a cone of the fan $\Sigma.$
It is enough to show that the element $\prod_{i \in I} [v_i]$
gets mapped to zero. We can write that:

\begin{equation}\nonumber
\begin{split}
& \psi_\alpha (\prod_{i \in I} [v_i])=\prod_{i \in I \cup I(\sigma(\alpha))} (1-R_i) 
\prod_{i \in I \setminus I(\sigma(\alpha))} 
\frac{\log  R_i }{1-R_i} \\
& \cdot \prod_{i \in I \cap I(\sigma(\alpha))} 
\log (e^{-2 \pi \sqrt{-1} \alpha_i} R_i)(1-R_i)^{-1}
\prod_{i \in I(\sigma(\alpha)) \setminus I} (1-R_i)^{-1}.
\end{split}
\end{equation}
The last three subfactors in the above product are all elements in 
the ring $(K\CC(\SSigma;\beta))_\alpha,$ and, since
the vectors $v_i$ with $i \in I \cup I(\sigma(\alpha))$ do not generate a 
cone of $\Sigma,$ the first subfactor is zero. We have
obtained that $\psi_\alpha ([v])=
\psi_\alpha (\prod_{i \in I} [v_i])=0,$ as 
required.

Hence, $\psi_\alpha$ descends to a morphism from 
$S_{\SSigma} \otimes_\ZZ \CC/(Z+S_{\SSigma}(\alpha))$ to 
$(K\CC(\SSigma;\beta))_\alpha.$ 
It is clear that $\phi_\alpha$ and $\psi_\alpha$ are
inverse maps, so the results follows.
\end{proof}

\begin{remark} Certainly, the isomorphism $\phi$ 
has the flavor of a Chern character.
It would be interesting to make this assertion
precise in a geometrically meaningful way. 
For a brief discussion of a similar issue, 
see \cite[Remarks 5.4, 5.5]{BH1}.
\end{remark}

\begin{example}\label{ex:ex}
Consider the stacky fan $\SSigma$
in $\ZZ^2$ determined by the
vectors $v_1=(1,0), v_2=(0,1), v_3=(-2,-1)$ and the obvious cones,
and $\beta=(a,b) \in  N \otimes \CC.$
The associated Deligne--Mumford stack is the 
stacky weighted projective space
$\PP(2,1,1).$

According to Proposition \ref{prop:kth1}, the maximum
ideals of the $K\CC(\SSigma; \beta)$ by solving
for $P=(y_1, y_2, y_3) \in (\CC^*)^3$ the following system
$$
(1-y_1)(1-y_2)(1-y_3)=0, \, y_1 y_3^{-2}=e^{2 \pi \sqrt{-1} \, a}, \, y_2 y_3^{-1}=e^{2 \pi \sqrt{-1} \, b}. 
$$
We obtain four solutions $P_i \in  (\CC^*)^3$
which for generic values $a,b \in \CC$ are distinct: 
\begin{equation*}
\begin{split}
P_1&=(1, e^{2 \pi \sqrt{-1} \, (-a/2+b)}, e^{2 \pi \sqrt{-1} \, (-a/2)}), \\
P_2&=(1, e^{2 \pi \sqrt{-1} \, (-a/2+b+1/2)}, e^{2 \pi \sqrt{-1} \, (-a/2+1/2)}), \\
P_3&=( e^{2 \pi \sqrt{-1} \, (a-2b)}, 1, e^{2 \pi \sqrt{-1} \, (-b)}), \\
P_4&=( e^{2 \pi \sqrt{-1} \, a}, e^{2 \pi \sqrt{-1} \, b},1).\\
\end{split}
\end{equation*}
Of course, this four points correspond to four possibly 
distinct elements in $\Bx(\SSigma; \beta).$

For any generic $\beta,$ the ring $K\CC(\SSigma;\beta)$ 
is semisimple of dimension equal to $4,$ which is 
the normalized volume of the polytope determined by the vectors
$v_i.$ The non-generic values of $\beta$ are hyperplanes
in $\CC^2$ where ``collisions" of some of the points $P_i$
occur.

For any $\beta=(a,b) \in \CC^2$ contained in a hyperplane of the form
$-a/2 + b \in \ZZ,$ we have that $P_1=P_3,$ while
along the hyperplanes $-a/2 +b +1/2 \in \ZZ,$ we have that
$P_2=P_3.$ 
For any $\beta=(a,b)$ contained in a hyperplane of the form
$-a/2 \in \ZZ,$ we have that $P_1=P_4,$ while
along the hyperplanes $-a/2 +1/2 \in \ZZ,$ we have that
$P_2=P_4.$ 
Moreover, for any $\beta=(a,b)$ in a hyperplane of the form
$b \in \ZZ,$ we have that $P_3=P_4.$

If we choose $\beta=(0,0),$ we have that $P_1=P_3=P_4,$
and $K\CC(\SSigma; \beta) \cong K_0 (\PP(2,1,1)) \otimes \CC
\simeq \CC[R_3^{\pm}]/(1-R_3)^3 (1+R_3).$ For $\beta=(0,1/2),$
we have that $P_1=P_4, P_2=P_3,$ and 
$K\CC(\SSigma; \beta) \cong \CC[R_3^{\pm}]/(1-R_3)^2 (1+R_3)^2.$
Of course, in both cases we have that $R_3$ is the class corresponding
to $v_3.$

\end{example}

\section{Application: Gamma series solutions to the 
better behaved GKZ hypergeometric system when
$\beta \in N \otimes \CC$} \label{sec:bbgkz}

As before, we assume that 
$\SSigma=(\Sigma, \{v_i\}_{1 \leq i \leq k})$ is
a simplicial stacky fan in the lattice $N.$
In this section, $\Aa$ denotes
the set $\{ v_1, v_2, \ldots, v_k \}
\subset N,$ 
and we allow for the possibility that 
some the elements of $\Aa$ do not
generate rays of the fan $\SSigma.$
We denote by $I(\SSigma) \subset
\{ 1, 2, \ldots k \},$ the set of indices $i$
such that $v_i$ generates a 
ray of $\SSigma.$ We also assume that there exists a group homomorphism
${\rm deg} : N \to \ZZ$ such that ${\rm deg}(v_i)=1$
for all $i, 1 \leq i \leq k,$ and that
the elements
$v_i,  1 \leq i \leq k,$ generate the lattice $N$ as 
a group.  We assume that the fan $\Sigma$ is 
induced by a regular triangulation of the convex
polytope $\Delta,$ the convex hull of 
the vectors $v_i.$ In particular, 
the support of the fan $\SSigma$ (denoted by $\Sigma$ so far
in this note) coincides with
the cone $K \subset N \otimes \RR$ over the polytope
$\Delta$ generated by the vectors $v_i.$
We assume that the fixed parameter
$\beta$ is in $N \otimes \CC.$

\begin{definition} For any parameters $\chi, \xi \in 
N \otimes \RR,$ the {\it shadow} module 
$\ZZ[\SSigma; \chi]_{\xi}$ is defined as 
the $R$-submodule of $\ZZ[\SSigma; \chi]$
generated by all $[n+\chi] \in \Sigma,$ with 
$n \in N,$ such that
$n + \chi + \epsilon \xi \in \Sigma$ for all 
sufficiently small $\epsilon >0.$
\end{definition}

The arguments used in \cite[Section 3]{Borisov} essentially
show that the following theorem holds:


\begin{proposition}\label{prop:regseq}
The ring $R=\CC[\SSigma],$ as well as the $R$-modules 
$\CC[\SSigma; \chi]$ and $\CC[\SSigma; \chi]_\xi$ 
are Cohen-Macaulay of dimension $d.$ 
Moreover,
for any basis $(g_1,\ldots, g_d )$ of $M={\rm Hom} (N, \ZZ),$ the elements
$$
Z_j=\sum_{i, i \in I(\SSigma)} g_j(v_i) [v_i]
$$
form a regular sequence in $R$ 
(and hence in $\CC[\SSigma; \chi]$ and
$\CC[\SSigma; \chi]_\xi$). 
\end{proposition}

\begin{corollary}\label{cor:mcm}
The quotients  
$
H\CC(\SSigma; 0)=
\CC[\SSigma]/Z \cdot \CC[\SSigma],
H\CC(\SSigma; \chi)=
\CC[\SSigma;\chi]/Z \cdot \CC[\SSigma; \chi],$ and 
$\CC[\SSigma;\chi]_\xi /Z \cdot \CC[\SSigma; \chi]_\xi,$
have dimension equal to the normalized volume of $\Delta$.
\end{corollary}

We now begin discussing the main theme of this section, 
namely the construction of $\Gamma$-series solutions 
to the better behaved GKZ system for $\beta \in N \otimes \CC.$

\begin{definition}\label{def:setL}
For each $\alpha=(\alpha_i) \in  \Bcx (\SSigma;\beta)$
and any $v \in K \cap N,$ we define the set $L(\alpha,v) \subset \CC^n$
of collections $l=(l_i)_{1 \leq i \leq k}$ such that
$$
\sum_{i=1}^k l_i v_i = \beta - v$$ 
and $l_i - \alpha_i$ are integers for all
$i, 1 \leq i \leq k.$
\end{definition}

Since the elements
$v_i,  1 \leq i \leq k,$ generate the lattice $N$ as 
a group, we see 
that, for a given $v \in K,$ any solution to the equation
$\sum_{i=1}^k l_i v_i = \beta -v$ such that
$l_i \in \ZZ$ if $i \notin I(\sigma)$ 
for a maximal cone $\sigma$
in $\Sigma,$ belongs to the
set $L(\alpha,v)$ for some $\alpha \in 
\Bcx (\sigma;\beta).$

We now choose a small enough $\delta >0$ such that
the results of Proposition \ref{prop:rel} and
Corollary \ref{cor:rel} hold. In particular, we have
an $R$--module isomorphism between 
$\CC[\SSigma;\beta]$ and 
$\CC[\SSigma;\Re \beta + \delta \Im \beta],$
and an induced triple one-to-one correspondence 
$\alpha \leftrightarrow \alpha_\delta \leftrightarrow c(\alpha_\delta)$ 
among the sets
$\Bcx (\SSigma;\beta) \subset \CC^k,$
$\Bcx (\SSigma;\Re \beta + \delta \Im \beta) \subset \RR^k$ and 
$\Bx (\SSigma;\Re \beta + \delta \Im \beta) \subset 
N \otimes \RR$ such that the corresponding support 
cones coincide:
$\sigma(\alpha)=\sigma(\alpha_\delta)=\sigma(c(\alpha_\delta)).$
In what follows, it will be convenient 
to use the notation 
$$\beta_\delta:= 
\Re \beta + \delta \Im \beta \in N \otimes \RR.$$

For a given $x=(x_1, \ldots, x_k)$ in $(\CC^\star)^k,$ and $v \in K \cap N,$
we introduce the formal $\Gamma$--series $\Phi_v(x)$
with values in the completion 
$\CC [\SSigma; \beta_\delta]^{\wedge}$
of the 
graded $R$--module $\CC [\SSigma; \beta_\delta]$
as follows
\begin{equation}\label{def:gamma1}
\Phi_v (x):=
\sum_{c(\alpha_\delta) \in \Bx  (\Sigma;\beta_\delta)} \sum_{l \in L(\alpha,v)}
\prod_{i=1}^k  \frac{x_i^{l_i + D_i}}
{\Gamma(l_i+ D_i+1)} \cdot [c(\alpha_\delta)],
\end{equation}
where
$$
D_i := [v_i] \ \text{if} \  i \in I(\SSigma), \ \text{and} \
D_i:=0, \ \text{otherwise,}$$ 
and 
$$
x_i^{l_i+D_i}:= e^{(l_i+D_i) (\log |x_i| + \sqrt{-1} \arg x_i)},$$
for a choice of $(\arg x_1, \ldots, \arg x_k) \in \RR^k.$

According to \cite[Proposition 2.12]{BH2}, for 
each $v \in K \cap N,$ the formal series 
$\Phi_v (x)$ 
induces a well defined map 
from a non-empty 
open set $U_{\Sigma}$ 
in $\CC^k$ to the completion $\CC[\SSigma, \beta_\delta]^{\wedge}$ of the
graded ring $\CC[\SSigma, \beta_\delta].$ 


\begin{lemma} \label{lemma:l1}
\begin{itemize}
\item[i)]
For any $c(\alpha_\delta) \in \Bx (\SSigma;\beta_\delta),$ $v \in \Sigma \cap N,$ 
$l \in L(\alpha,v)$ and $x \in \CC^k,$ the product
$$
\prod_{i=1}^k  \frac{x_i^{l_i + D_i}}
{\Gamma(l_i+ D_i+1)} \cdot 
[c(\alpha_\delta)],
$$
belongs to the $R$-submodule
$\CC [\SSigma; \beta_\delta]_{\Re \beta}
\subset \CC [\SSigma; \beta_\delta].$  

\item[ii)] For any element $[w] \in 
\CC [\SSigma; \beta_\delta]_{\Re \beta},$ there exists 
a maximal cone $\sigma$ of $\SSigma,$ a unique element
$c(\alpha_\delta) \in \Bx (\sigma; \beta_\delta),$ and elements 
$v \in K \cap N,$  $l=(l_i) \in L(\alpha,v),$
such that $l_i =0$ if $i \notin I(\sigma),$ and 
$$
w \in c(\alpha_\delta) +  \sum_{i, l_i \in \ZZ_{< 0}} v_i + 
\sum_{i \in I(\sigma)} \ZZ_{\geq 0} v_i.
$$
\end{itemize}
\end{lemma}

\begin{proof} {\it i)}
Note first that the product is zero, unless
there exists a maximal cone $\sigma$ 
such that $I(\sigma(\alpha_\delta)) \subset I(\sigma)$ 
and $I(\sigma)$ contains all
$i$ with $l_i \in \ZZ_{< 0}.$ 
Let $\sigma$ be 
such a maximal cone. 
For any $i, i \notin I(\sigma),$ 
we have that $l_i \in \ZZ_{\geq 0}.$

The product  
lies in the the $R$-submodule of 
$\CC [\SSigma; \beta_\delta]$ 
generated by $[w]$ where 
$$
w = \sum_{i=1}^k (\alpha_\delta)_i v_i + \sum_{i, l_i \in \ZZ_{<0}} v_i=
n +\beta_\delta,$$ 
for some $n \in N.$ Since 
$l_i - \alpha_i$ are integers for all
$i,$ and $\sigma(\alpha)=\sigma(\alpha_\delta),$
it is certainly the case that 
$l_i \in \ZZ$ if and only if $i \notin I(\sigma(\alpha_\delta)).$

For any $\epsilon>0$ we can write that
$$
w +\epsilon (\Re \beta -v)  = w  +
\epsilon \sum \Re l_i v_i 
$$
$$ 
=\sum_{i, i \in I(\sigma(\alpha_\delta))} ((\alpha_\delta)_i+\epsilon \Re l_i) v_i
+\sum_{i, l_i \in \ZZ_{<0}} (1+\epsilon l_i) v_i
+\sum_{i, l_i \in \ZZ_{\geq 0}} \epsilon l_i v_i.
$$
But $0< (\alpha_\delta)_i < 1$
for any $i \in I(\sigma(\alpha_\delta)),$
so we indeed have that $w + \epsilon (\Re \beta -v)$ belongs to the 
cone $\Sigma$ for sufficiently small $\epsilon >0.$ Since
$v \in \Sigma,$ we also have that $w + \epsilon \Re \beta 
\in \Sigma,$ so 
$[w] \in \CC [\SSigma; \beta_\delta]_{\Re \beta}.$

{\it ii)} For any $[w] \in 
\CC [\SSigma; \beta_\delta],$ there exists
$n \in N$ such that $w=n + \beta_\delta.$
We choose a maximal cone $\sigma$ 
such that $w + \epsilon \beta_\delta \in \sigma$
for all sufficiently small $\epsilon >0.$
As explained in Remark \ref{rem:11}
since $\beta_\delta \in N \otimes \RR,$
there exists a uniquely defined $c(\alpha_\delta)
=\sum_{i=1}^k (\alpha_\delta)_i v_i \in 
\Bx (\SSigma;\beta)$ contained in $\sigma$ such 
that
\begin{equation}\label{eq:w}
w= n+ \beta_\delta=
c(\alpha_\delta) + \sum_{i \in I(\sigma)} 
p_i v_i = \sum_{i \in I(\sigma)} ((\alpha_\delta)_i + p_i) v_i,
\end{equation}
with $p_i \in \ZZ_{\geq 0}.$ Under the correspondence 
between $\alpha_\delta \in \Bcx (\SSigma; \beta_\delta) \subset \RR^k$ and 
$\alpha \in \Bcx (\SSigma; \beta) \subset \CC^k,$ we have that 
$(\alpha_\delta)_i = \{ \Re \alpha_i + \delta \Im \alpha_i\}.$ 
In particular, this means that
$$
n+\beta \in \sum_{i \in I(\sigma)} (\alpha_i +\ZZ) \, v_i
$$
For any $n \in N,$ 
there exists a unique $v \in \Bx (\sigma;0) \subset N,$
such that $n \in  - v + \sum_{i \in I(\sigma)} \ZZ v_i.$ We conclude that
we are able to write
$$
\beta-v = \sum_{i=1}^k l_i v_i,
$$
with $l=(l_i) \in L(\alpha,v)$ and $l_i =0$ if $i \notin I(\sigma).$ 
By comparing this expression with equation (\ref{eq:w}),
it is easy to see that
$w + \epsilon \beta_\delta \in \sigma$
for all sufficiently small $\epsilon >0,$ implies that
$p_i \in \ZZ_{>0}$ for all $i$ with $l_i \in \ZZ_{<0},$ 
and this ends the proof of the lemma.
\end{proof} 
As the first part of this lemma indicates, the series 
$\Phi_v (x)$ takes values 
in the completion 
$(\CC [\SSigma; \beta_\delta]_{\Re \beta})^{\wedge}$
of the 
graded $R$--module $\CC [\SSigma; \beta_\delta]_{\Re \beta}.$

\begin{theorem}
For any linear map
$$h : \CC [\SSigma; \beta_\delta]_{\Re \beta}
/ (Z_1, \ldots, Z_d) \CC [\SSigma; \beta_\delta]_{\Re \beta}
\to \CC,$$ 
the sequence of functions 
$(h \circ \Phi_{v}(x_1,\ldots,x_v))_{v \in K \cap N}$ satisfies the better 
behaved GKZ hypergeometric equations (\ref{eq:gkz1}) and (\ref{eq:linear})
on $U_\SSigma$ 
corresponding to the set $\mathcal A$ and parameter $\beta \in N \otimes \CC.$
\end{theorem}

\begin{proof} For the set of equations (\ref{eq:gkz1}), 
note that taking the $j$-th partial derivative 
of the summand over $(l_i)_{1 \leq i \leq k} \in
L(\alpha,v)$ in the series $\Phi_v,$
replaces it by a summand over 
$(l_i - \delta^i_j)_{1 \leq i \leq k}.$ It is then
enough to note that 
$$
L (\alpha, v) - (\delta^i_j)_{1 \leq i \leq k} = L(\alpha, v + v_j).
$$

For the set of equations (\ref{eq:linear}), our convention
that
$D_j=0$ if $j \notin I(\SSigma),$ 
implies that
$$
\sum_{i=1}^k g(v_i) D_i=
\sum_{i, i \in I(\SSigma)} g(v_i) D_i,
$$
for any $g \in M={\rm Hom} (N, \ZZ).$ Hence
$$
\big( - g(\beta - v) + \sum_{i=1}^k g(v_i) x_i \partial_i \big) \Phi_{v}
=\big(\sum_{i=1}^k g(v_i) D_i \big) \Phi_{v}.
$$
The previous lemma shows that the series $\Phi_v$ 
takes values in $\CC [\SSigma; \beta_\delta]_{\Re \beta},$ so the
result follows after we
observe that $\sum_{i=1}^k g(v_i) D_i$ is a linear 
combination of the $Z_j$'s.
\end{proof}

Accordingly,
it is convenient to view each 
$\Gamma$--series 
$\Psi_v(x), v \in K \cap N,$ as a map
from the non-empty open set $U_{\SSigma}$ in $\CC^k$ to the finite dimensional vector space 
$\CC [\SSigma; \beta_\delta]_{\Re \beta}/(Z_1, \ldots, Z_d) \CC [\SSigma; \beta_\delta]_{\Re \beta}.$ 
We now prove the required linear independence result.

\begin{proposition}
If $h: \CC [\SSigma; \beta_\delta]_{\Re \beta}
/ (Z_1, \ldots, Z_d) \CC [\SSigma; \beta_\delta]_{\Re \beta}
\to \CC$
is a linear map such that $h \circ \Phi_{v} (x)=0,$ for all $v \in K \cap N$
and any $x \in U_\SSigma,$
then $h=0.$
\end{proposition}

\begin{proof} The proof is very similar to the proof 
of \cite[Proposition 2.19]{BH2}. We include it here since 
the context and the notation are
slightly changed.
Assume that 
there exists some 
$x \in \CC [\SSigma; \beta_\delta]_{\Re \beta}
/ (Z_1, \ldots, Z_d) \CC [\SSigma; \beta_\delta]_{\Re \beta}$
such that $h(x)\not=0.$ Let $P$ be the largest degree
of such an element. Furthermore, we choose  
a representative $[w]$ mod 
$Z \CC [\SSigma; \beta_\delta]_{\Re \beta}$
for $[x].$ According to part $ii)$ 
of Lemma \ref{lemma:l1}, there exists 
a maximal cone $\sigma$ and a unique 
$c(\alpha_\delta) \in \Bx (\sigma; \beta_\delta)$
an  element $v \in K \cap N,$ and a relation $l \in L(\alpha,v)$
such that
\begin{equation}\label{eql}
w= c(\alpha_\delta)+ \sum_{i, l_i \in \ZZ_{<0}} v_i + \sum_{i \in I(\sigma)} n_i v_i 
\end{equation}
with $n_i \in \ZZ_{\geq 0}.$ 

For each $i$ such that $v_i$ with $i \in I(\SSigma),$ 
consider the loop of the form $x_i(t)= \epsilon \exp(2 \pi \sqrt{-1} \, t), 
x_j(t)=\epsilon,$
$j\not=i,$ $0 \leq t \leq 1,$ 
with $\epsilon >0 $ a small real positive number. 
The action of the induced monodromy operator $T_i$
on the $\Gamma$-series $\Phi_{v}$ is given by $\exp(D_i)$.
Since $D_i$ is nilpotent in 
$\CC [\SSigma; \beta_\delta]_{\Re \beta}
/ Z \cdot \CC [\SSigma; \beta_\delta]_{\Re \beta},$
there is a polynomial
$g(T_i)$ such that $g(T_i) \Phi_{v} = D_i \Phi_{v}$, for every $i \in I(\SSigma).$
Hence 
$$
\prod_i g(T_i)^{n_i} \Phi_{v}(x)=
\prod_i D_i^{n_i} \Phi_{v}(x).
$$
Since $h \circ \Phi_{v}(x)=0$ and we have analytically continued $\Phi_{v},$
we also have that
$$
h (\prod_i D_i^{n_i} \Phi_{v}(x)) =0.$$ 

The definition of the $\Gamma$--series  
$\Phi_{v}(x)$ and the fact that $D_i=[v_i]$ are nilpotent in 
$\CC [\SSigma; \beta_\delta]_{\Re \beta}
/ Z \cdot \CC [\SSigma; \beta_\delta]_{\Re \beta}$
shows that any induced GKZ solution 
can be written as the product of a monomial in the variables
$x_i$ and an element of $\CC[u_j^{-1}, \log u_j] [[u_j]]$
where $u_j, 1 \leq j \leq n- {\rm rank} N,$ invariant variables
under the action of the character torus ${\rm Hom} (L, \CC^\times),$
where $L$ is the lattice of integral relations $(l_i)$
such that $\sum_{i=1}^k l_i v_i=0.$

Hence, 
in order to obtain the contradiction it is enough 
to show that, for $l \in L(\alpha,v)$ used in formula (\ref{eql}),
the Fourier coefficient of $x^l= \prod x_i^{l_i}$
in the expansion of $h (\prod_i D_i^{n_i} \Phi_{v}(x))$
is non-zero. Indeed, this coefficient 
is given by
$$
h(\prod  D_i^{n_i} \cdot [c(\alpha_\delta)] \cdot \prod \frac{1}{\Gamma(l_i+D_i+1)}).
$$
Notice that the terms that occur in the expansion of the 
expression in the argument 
of $h$ have degree at least $P,$
while $[w]$ is the only element of that degree
that occurs and its coefficient is nonzero.
Since $h([w]) \not= 0,$
the maximal property of $P$ implies that 
the coefficient of $x^l$ in the expansion 
of $h (\prod_i D_i^{n_i} \Phi_{v}(x))$
is indeed non-zero. 
This ends the proof of the 
linear independence result.
\end{proof} 

According to \cite[Corollary 3.9]{BH11},
the dimension of the space of solutions to the
better behaved GKZ system is exactly ${\rm vol} (\Delta),$ so
we conclude that the formal $\Gamma$--series produces the expected 
number of linearly independent analytic solutions
to the 
better behaved GKZ system.
More precisely, we have constructed the following
``topological mirror symmetry" map:

\begin{theorem}\label{thm:gammanontor}
The map
\begin{equation*}
\begin{split}
\big(\CC [\SSigma; \beta_\delta]_{\Re \beta}
/ Z \cdot \CC [\SSigma; \beta_\delta]_{\Re \beta}\big)^\vee
&\to {\mathcal Sol} (U_{\SSigma}) \\
f &\to (f \circ \Phi_{v})_{v \in K \cap N}
\end{split}
\end{equation*}
produces a complete system of 
${\rm vol} (\Delta)$ linearly independent 
solutions to the better behaved GKZ system
which are analytic in $U_{\SSigma}.$
\end{theorem}

\begin{remark} The result of this theorem should be compared
with the results of \cite[Corollary 2.21]{BH2}
which employ the ``leading term module'' $M(\beta)$
associated to the usual GKZ system for $\beta \in N.$
The modules 
$M(\beta)$ are useful algebraic tools defined in loc. cit.,
but the dimensions of $M(\beta)/ZM(\beta)$ are quite hard to control 
for general values of  $\beta.$ In contrast, 
the shadow modules $\CC[\SSigma; \beta_\delta]_{\Re \beta}$ 
are maximal Cohen-Macaulay. It is
possible to show 
that, for $\beta \in N,$ we have that 
$\CC[\SSigma; \beta]_{\beta} = \sum_{v \in K \cap N} M(\beta -v).$
\end{remark}

\begin{remark} It is possible to give a $K$--theoretic combinatorial 
description in the spirit of the discussion from
Section \ref{sec:kth} of the shadow modules 
$\CC [\SSigma; \beta_\delta]_{\Re \beta}
/ Z \cdot \CC [\SSigma; \beta_\delta]_{\Re \beta}.$
Somewhat surprisingly, in our approach such a description is related 
to the notion of a $\gamma$--deformed version of the
set $\Bx (\SSigma; \beta)$ mentioned in 
Remark \ref{rem:boxlimit} which we do not pursue in this 
note. It suffices to say that combining such a description
with the tools developed in \cite[Section 3]{BH2} provides
a $K$--theoretical version of the topological mirror symmetry 
map described above. 
\end{remark}

\section*{Acknowledgements} I would like to thank Lev Borisov
for very useful comments and discussions. Many of the ideas presented 
in this note have been developed in the course of our
joint collaboration. I would like to thank the organizers of the 
Maximilian Kreuzer Memorial Conference for the opportunity to 
speak at the conference in June 2011, and for the invitation to submit
a contribution to this volume. This work is dedicated to the memory of
Maximilian Kreuzer whose irresistible enthusiasm 
for subjects at the interface of string theory
and algebraic geometry remains a great source of inspiration. 
The topics of this note, in particular toric geometry and GKZ systems, 
are among those which benefited greatly from his ideas.


\begin{thebibliography}{9} 


\bibitem{Batyrev}
V.~V. Batyrev, \emph{Variations of mixed Hodge structure of affine
hypersurfaces in algebraic tori}, Duke Math. J. {\bf 69} (1993),
349--409.


\bibitem{Borisov}
L.~A.~Borisov, \emph{String cohomology of a toroidal singularity}, 
J. Algebraic Geom. {\bf 9} (2000), no. 2, 289--300, {\tt math.AG/9802052}.



\bibitem{BCS} 
L.~A.~Borisov, L.~Chen, G.~Smith, \emph{The orbifold Chow ring 
of toric Deligne-Mumford stacks},
J. Amer. Math. Soc.  {\bf 18}  (2005),  no. 1, 193--215, 
{\tt math.AG/0309229}.


\bibitem{BH11}
L.~A. Borisov, R.~P. Horja, {\em 
On the better behaved version of the GKZ hypergeometric  
system}, preprint, 2010; {\tt arXiv:1011.5720 [math.AG]}.


\bibitem{BH1}
L. A.~Borisov, R.~P. Horja, \emph{On the $K$-theory of smooth
toric DM stacks}, Snowbird lectures on string geometry, 21--42, Contemp. Math., 
{\bf 401}, Amer. Math. Soc., Providence, RI, 2006, {\tt math.AG/0503277}.


\bibitem{BH2}
L.~A. Borisov, R.~P. Horja, {\em 
Mellin-Barnes integrals as Fourier-Mukai transforms}, Adv. Math. 
{\bf 207} (2006), no. 2, 876-927, {\tt math.AG/0510486}.




\bibitem{GKZ}
I.~M.~Gelfand, M.~M.~Kapranov, A.~V.~Zelevinsky,
\emph{Hypergeometric functions and toric varieties}, 
(Russian)  Funktsional. Anal. i Prilozhen.  {\bf 23}  (1989),  no. 2, 12--26;  
translation in  Funct. Anal. Appl.  
{\bf 23}  (1989),  no. 2, 94--106. 


\bibitem{HLY}
S.~Hosono, B.H.~Lian, S.-T.~Yau, 
\emph{ Maximal degeneracy points of GKZ systems}, 
J. Amer. Math. Soc. {\bf 10} (1997), no. 2, 427--443; {\tt alg-geom/9603014}.


\bibitem{MMW}
L.~F.~Matusevich, E.~Miller, U.~Walther, \emph 
{Homological methods for hypergeometric families}, 
J. Amer. Math. Soc. {\bf 18} (2005), no. 4, 919-941; {\tt math.AG/0406383}.

\bibitem{SST}
M.~Saito, B.~Sturmfels, N.~Takayama, \emph{Gr\"{o}bner deformations
of hypergeometric differential equations}, Springer, 2000. 


\bibitem{stienstra}
J.~Stienstra, \emph{Resonant hypergeometric systems and mirror symmetry},
Integrable systems and algebraic geometry (Kobe/Kyoto, 1997),  412--452, 
World Sci. Publishing, River Edge, NJ, 1998, {\tt math.AG/9711002}.




\end{thebibliography}
\end{document}